\documentclass[reqno]{article}
\usepackage{DejaVuSerif}
\usepackage{noto-sans}
\usepackage{noto-mono}
\usepackage[T2A]{fontenc}
\usepackage[utf8]{inputenc}
\usepackage{color}
\usepackage{amsthm}
\usepackage{amssymb}
\usepackage[pdfusetitle,
 bookmarks=true,bookmarksnumbered=false,bookmarksopen=false,
 breaklinks=false,pdfborder={0 0 1},backref=false,colorlinks=true]
 {hyperref}

\makeatletter

\DeclareTextSymbolDefault{\textquotedbl}{T1}

\newcommand{\lyxaddress}[1]{
	\par {\raggedright #1
	\vspace{1.4em}
	\noindent\par}
}
\theoremstyle{remark}
\newtheorem{rem}{\protect\remarkname}
\theoremstyle{remark}
\newtheorem{claim}{\protect\claimname}
\theoremstyle{plain}
\newtheorem{thm}{\protect\theoremname}
\theoremstyle{plain}
\newtheorem{lem}{\protect\lemmaname}
\theoremstyle{plain}
\newtheorem{cor}{\protect\corollaryname}

\usepackage{pifont}
\usepackage{amstext}
\usepackage[tbtags]{amsmath}
\renewcommand{\-}{-}

\makeatother

\usepackage[english,russian]{babel}
\providecommand{\claimname}{Утверждение}
\providecommand{\corollaryname}{Следствие}
\providecommand{\lemmaname}{Лемма}
\providecommand{\remarkname}{Замечание}
\providecommand{\theoremname}{Теорема}

\begin{document}
\title{On mixing actions of locally compact groups\thanks{Данное исследование выполнено в рамках государственного задания в
сфере научной деятельности Министерства науки и высшего образования
РФ на тему \textquotedbl Модели, методы и алгоритмы искусственного
интеллекта в задачах экономики для анализа и стилизации многомерных
данных, прогнозирования временных рядов и проектирования рекомендательных
систем\textquotedbl , номер проекта FSSW-2023-0004.}}
\author{S. V. Tikhonov}
\maketitle

\lyxaddress{ФГБОУ ВО «РЭУ им. Г.В. Плеханова», г.Москва\\ Губкинский университет,
г. Москва}
\begin{abstract}
In this paper, we construct the leash-metric that transforms the set
of (partially) mixing actions of a Hausdorff locally compact group
with a countable neighborhood base into a complete separable metric
space.
\end{abstract}

\section*{введение}
В работе \cite{Tikh2013JDCS} изучены вопросы метризации пространства перемешивающих действий дискретных групп.
перемешивающих действий дискретных групп. Построенная там поводок-метрика
позволяет использовать для перемешивающих действий понятие типичности,
неприменимое при использовании стандартной слабой топологии. Работы
\cite{Tik11,Ryz20} показывают, что те же вопросы представляют интерес
как для несчетных групп, так и для частично перемешивающих действий.
В последнем случае результаты имеют приложения, в частности, к спектральной
теории преобразований. Еще одной мотивировкой работы является то,
что на практике слабая топология для групп может быть задана различными
метриками. 

Мы обобщаем результат работы \cite{Tikh2013JDCS} на более широкий
класс метрик, групп и действий, включающий метрики и действия из работ
\cite{Tik11} и \cite{Ryz20}. 

Для формулировки основного результата нам понадобятся некоторые определения. 

Пусть $\left(X,\Sigma,\mu\right)$ --- сепарабельное пространство
Лебега и $\left\{ A_{i}\right\} $ --- счетный набор множеств порождающий
$\sigma$-алгебру $\Sigma$. Мера $\mu$ непрерывна и нормирована.
Через $\mathcal{A}$ обозначим множество обратимых, сохраняющих меру
преобразований пространства $\left(X,\Sigma,\mu\right)$. Два преобразования
считаются совпадающими, если они отличаются лишь на множестве нулевой
меры.

На $\mathcal{A}$ определены две метрики, задающие слабую топологию:
\[
d\left(T,S\right)=\sum_{i\in\mathbb{N}}\frac{1}{2^{i}}\left(\mu\left(TA_{i}\bigtriangleup SA_{i}\right)+\mu\left(T^{-1}A_{i}\bigtriangleup S^{-1}A_{i}\right)\right),
\]
и 
\[
a\left(T,S\right)=\sum_{i,j\in\mathbb{N}}\frac{1}{2^{i+j}}\left|\mu\left(TA_{i}\cap A_{j}\right)-\mu\left(SA_{i}\cap A_{j}\right)\right|.
\]

Пусть $\mathcal{G}$ --- локально компактная хаусдорфова группа со
счетной базой окрестностей, $\Gamma$ --- ее неограниченное подмножество,
а $\left\{ K_{i}\right\} $ --- не более чем счетный набор компактов
с непустой внутренностью, покрывающий все образующие группы $\mathcal{G}$. 

\emph{Действием} группы $\mathcal{G}$ называется набор преобразований
$T=\left\{ T^{g}\right\} _{g\in\mathcal{G}}$, для которого $T^{g}T^{h}=T^{gh}$
при всех $g,h\in\mathcal{G}$ и отображение $g\mapsto\mu\left(T^{g}A\cap B\right)$
непрерывно для любых $A,B\in\Sigma$. 

Действие $T$ группы $\mathcal{G}$ называется\emph{ Г\-перемешивающим},
если для любых $A,B\in\Sigma$, имеем $\mu\left(T^{g}A\cap B\right)\to\mu\left(A\right)\mu\left(B\right)$
при $g\in\Gamma,\left|g\right|\to\infty$. 

Основной результат работы (теорема \ref{GL}) следующий: 

\emph{Множество $\mathcal{M}_{\mathcal{G},\Gamma}$ всех $\Gamma$-перемешивающих
$\mathcal{G}$-действий является полным сепарабельным пространством
относительно метрики 
\[
\mathrm{m}_{\mathcal{G},\Gamma}\left(T,S\right)=\sum_{i}\frac{1}{2^{i}}\sup_{g\in K_{i}}\mathrm{d}\left(T^{g},S^{g}\right)+\sup_{g\in\Gamma}\mathrm{a}\left(T^{g},S^{g}\right).
\]
}
\begin{rem}
Существование набора $\left\{ K_{i}\right\} $ гарантируется структурой
группы, так как она имеет счетное компактное покрытие. Однако, выбирать
набор можно по\-разному. Для компактно порожденной группы достаточно
одного элемента $K_{1}$. Такая метрика использовалась, например,
в работе \cite{Tik11}. 

Для $\mathbb{R}$\-действий (потоков) в качестве $K_{1}$ обычно
рассматривается отрезок $\left[0,1\right]$, для $\mathbb{R}^{n}$\-действий
--- единичный куб $\left[0,1\right]^{n}$.

Счетный набор $\left\{ K_{i}\right\} $ используется, например, для
определения метрики на действиях общей дискретной группы $\mathcal{G}$,
когда каждое множество $K_{i}$ состоит из одного элемента $\mathcal{G}$.
Такая метрика, рассматривается в \cite{GlWeissT} и \cite{Tikh2013JDCS}.
\end{rem}
\begin{rem}
Элементы множества $\mathcal{M}_{\mathcal{G},\mathcal{G}}$ называются
(сильно) перемешивающими действиями группы $\mathcal{G}$. Теорема
предоставляет метрику для этого пространства. Однако, иногда сильно
перемешивающие действия можно метризовать проще, с помощью дискретного
множества $\Gamma$. Например, все $\mathbb{Z}$-перемешивающие действия
группы $\mathbb{R}$ являются сильно перемешивающими и для них более
удобной в применениях будет метрика $\mathrm{m}_{\mathbb{R},\mathbb{Z}}\left(T,S\right)$,
а не $\mathrm{m}_{\mathbb{R},\mathbb{R}}\left(T,S\right)$. 
\end{rem}
\begin{rem}
В работе \cite{Tik11} рассматривалась $\mathrm{m}_{\mathcal{H},\Gamma}$
метрика для действий некоммутативной группы $\mathcal{H}$ с двумя
образующими $g_{0},g_{1}$ такими, что $g_{0}^{n}=e$ и 
\[
g_{1},g_{0}^{-1}g_{1}g_{0},...,g_{0}^{1-n}g_{1}g_{0}^{n-1},
\]
попарно коммутируют. В качестве $\Gamma$ бралась подгруппа, порожденная
элементом $\left(g_{0}g_{1}\right)^{n}$. 
\end{rem}
В основном поводок\-метрика используется для установления типичных
свойств перемешивающих действий и преобразований. Напомним, что свойство
$P$ в полном сепарабельном метрическом пространстве называется \emph{типичным},
если им обладают элементы некоторого всюду плотного $G_{\delta}$\-множества
(то есть счетного пересечения всюду плотных открытых множеств). В
этом случае говорят, что \textquotedbl типичный элемент обладает
свойством $P$\textquotedbl . Применения слабой топологии для изучения
типичных свойств перемешивающих действий недостаточно. В пространствах
$\mathcal{M}_{\mathcal{G},\Gamma}$, напротив они широко исследуются.
Известно большое количество фактов о типичных перемешивающих преобразованиях,
см. \cite{Ryz24,Bashtanov2013,Tik07} (например, они кратно перемешивают),
для типичного потока установлено, что он имеет ранг 1, см. \cite{Ryz20}.
Для действий других групп установлены лишь разрозненные факты. В частности,
для типичного $\Gamma$\-перемешивающего действия $T$ упомянутой
выше группы $\mathcal{H}$, перемешивающее преобразование $T^{\left(g_{0}g_{1}\right)^{n}}$
обладает однородным спектром кратности $n$.

Общее исследование типичности требует хороших аппроксимационных инструментов.
Не для всех исследуемых групп они имеются даже в случае слабой топологии.
Однако, одно типичное свойство мы докажем в общем случае.

Множество действий вида $U^{-1}TU$, где $U\in\mathcal{A}$, называется
\emph{орбитой} точки $T$.

Типичность элементов полного сепарабельного метрического пространства
с плотной орбитой называется \emph{слабым рохлинским свойством}. Мы
покажем, что пространства $\mathcal{M}_{\mathcal{G},\Gamma}$ этим
свойством обладают.

Свойство $P$ называется \emph{динамическим}, если оно сохраняется
при сопряжениях. При наличии слабого рохлинского свойства, элементы,
обладающие свойством $P$ (при выполнении некоторых необременительных
требований) либо типичны, либо являются множеством первой категории.

Структура работы следующая.

В первой части работы вводится поводок\-метрика и доказывается теорема
о полноте и сепарабельности пространства $\mathcal{M}_{\mathcal{G},\Gamma}$.

Вторая часть работы посвящена возникающим в пространстве $\mathcal{M}_{\mathcal{G},\Gamma}$
топологиям, содержит ряд технических результатов о вспомогательных
метриках и предметриках. Описана база топологии, не использующая метрики
$\mathrm{m}$. 

В третьей часть работы для пространства $\mathcal{M}_{\mathcal{G},\Gamma}$
доказано слабое рохлинское свойство.

\section{Метрики}

\subsection*{Общие обозначения и замечания}

Для множества всех действий группы $\mathcal{G}$ будем использовать
обозначение $\mathcal{A}_{\mathcal{G}}$. 

Группу $\mathcal{G}$ будем считать группой по умножению с нейтральным
элементом $e$.

В некоторых утверждениях будут использоваться сразу несколько метрик
(или предметрик) в различных пространствах. Поэтому мы будем использовать
префикс \textquotedbl$\mathrm{m}-$\textquotedbl , если объект
рассматривается в метрике $\mathrm{m}$, \textquotedbl$\mathrm{d}-$\textquotedbl ,
если в метрике $\mathrm{d}$ и тому подобное. Например, $\mathrm{m}$\-замыкание
означает замыкание в метрике $\mathrm{m}$. Также будет использоваться
обозначение $\left(\mathrm{h},T,\varepsilon\right)$ для $\varepsilon$\-окрестности
действия $T$ в предметрике $\mathrm{h}$. 

Значок $\overset{\varepsilon}{\sim}$ означает, что величины отличаются
меньше чем на $\varepsilon$. 

Метрику $\mathrm{a}$ можно рассматривать как метрику слабой операторной
топологии. Действительно, каждая мера $\mu\left(TA\cap B\right)$
равна скалярному произведению $\left\langle T\chi_{A},\chi_{B}\right\rangle $,
в котором $T$ --- купмановский оператор, связанный с одноименным
преобразованием, а $\chi_{A}$ и $\chi_{B}$ --- индикаторы множеств
$A$ и $B$.

Тогда 
\[
\mathrm{a}\left(T,S\right)=\sum_{i,j}\frac{1}{2^{i+j}}\left|\left\langle T\chi_{A_{i}},\chi_{A_{j}}\right\rangle -\left\langle S\chi_{A_{i}},\chi_{A_{j}}\right\rangle \right|.
\]
Записанное справа выражение --- метрика слабой операторной топологии.
Вследствие такой двойственности метрики $\mathrm{a}$, мы будем ее
применять не только к преобразованиям, но и к операторам и, более
того, к паре (преобразование\-оператор). В последнем случае преобразование
заменяется соответствующим унитарным оператором. Например, преобразование
$T$ перемешивает, если 
\[
\mathrm{a}\left(T^{i},\Theta\right)\to0,
\]
при $i\to\infty$ (здесь и далее $\Theta$ --- ортопроектор на подпространство
констант в $L_{2}\left(X\right)$).

Кроме метрики $\mathrm{a}$ на пространстве унитарных операторов можно
также рассматривать и некоторые построенные с ее помощью метрики и
предметрики. 

Все рассматриваемые метрики и предметрики зависят от некоторого плотного
в $\Sigma$ набора множеств $\left\{ A_{i}\right\} _{i\in\mathbb{N}}$.
Плотность означает, что для любых $\varepsilon>0$ и $A\in\Sigma$
существует такое $i\in\mathbb{N}$, что 
\[
\mu\left(A_{i}\bigtriangleup A\right)<\varepsilon.
\]

Набор $\left\{ A_{i}\right\} $ фиксирован на протяжении всей работы.
Также фиксирован не более чем счетный набор компактных множеств с
непустой внутренностью $\left\{ K_{i}\right\} $, покрывающий некоторое
множество, порождающее $\mathcal{G}$.

Последовательность предметрик $\left\{ \mathrm{b}_{n}\right\} $ назовем
\emph{возрастающей}, если для любых элементов пространства $T$ и
$S$, 
\[
n<m\Rightarrow\mathrm{b}_{n}\left(T,S\right)\leq\mathrm{b}_{m}\left(T,S\right).
\]

Предметрику $\mathrm{b}$ назовем \emph{равномерным пределом} последовательности
$\left\{ \mathrm{b}_{n}\right\} $, если для любого $\varepsilon>0$,
существует такое число $n$, что 
\[
n<m\Rightarrow\forall T,S:\mathrm{b}_{m}\left(T,S\right)\overset{\varepsilon}{\sim}\mathrm{b}\left(T,S\right).
\]

\begin{claim}
Пусть возрастающая последовательность предметрик $\left\{ \mathrm{b}_{n}\right\} $
равномерно сходится к предметрике $\mathrm{b}$. Тогда, $\mathrm{b}$\-топология
задается базой окрестностей 
\[
\left\{ \left(\mathrm{b}_{n},T,\varepsilon\right)\right\} _{T,n,\varepsilon}.
\]
\end{claim}
\begin{proof}
Очевидно, что для любых $T,S$ и $n$ выполнено неравенство 
\[
\mathrm{b}_{n}\left(T,S\right)\leq\mathrm{b}\left(T,S\right).
\]
Тогда, $\left(\mathrm{b},T,\varepsilon\right)\subset\left(\mathrm{b}_{n},T,\varepsilon\right)$
для всех $T,n,\varepsilon>0$. Следовательно, множества $\left(\mathrm{b}_{n},T,\varepsilon\right)$
являются открытыми в $\mathrm{b}$\-топологии. Покажем, что они образуют
базу. 

Так как последовательность $\left\{ \mathrm{b}_{n}\right\} $ сходится
равномерно, существует такое $n$, что 
\[
\forall S:\mathrm{b}\left(T,S\right)-\mathrm{b}_{n}\left(T,S\right)<\frac{\varepsilon}{2}.
\]
Для $S\in\left(\mathrm{b}_{n},T,\frac{\varepsilon}{2}\right)$, имеем
\[
\mathrm{b}\left(T,S\right)<\mathrm{b}_{n}\left(T,S\right)+\frac{\varepsilon}{2}<\varepsilon.
\]
Таким образом, $\left(\mathrm{b}_{n},T,\frac{\varepsilon}{2}\right)\subset\left(\mathrm{b},T,\varepsilon\right)$. 
\end{proof}
Мы будем рассматривать два вида топологий --- слабую и поводок\-топологию
и их различные метризации.

\subsection*{Метрики и предметрики пространства $\mathcal{A}$}

Доказательства приведенных ниже фактов о пространстве $\mathcal{A}$
имеются в \cite{Halmosh} и \cite{Tik07}.

В работе не понадобится стандартное определение базы слабой топологии.
Известно, что эта топология метризуема любой из метрик $\mathrm{a}$
и $\mathrm{d}$, следовательно можно использовать порожденные ими
базы. 

Ниже перечислено несколько известных свойств пространства $\mathcal{A}$: 
\begin{itemize}
\item Метрика $\mathcal{A}$ является равномерным пределом возрастающей
последовательности предметрик 
\[
\mathrm{a}_{n}\left(T,S\right)=\sum_{i\leqslant n,j\leqslant n}\frac{1}{2^{i+j}}\left|\mu\left(TA_{i}\cap A_{j}\right)-\mu\left(SA_{i}\cap A_{j}\right)\right|,
\]
при $n\rightarrow\infty$ (грубая оценка показывает, что $\mathrm{a}\overset{\frac{1}{2^{n-1}}}{\sim}\mathrm{a}_{n}$).
Поэтому, слабая топология может быть задана базой окрестностей 
\[
\left\{ \left(\mathrm{a}_{n},T,\varepsilon\right)\right\} _{T\in\mathcal{A},n\in\mathbb{N},\varepsilon>0}.
\]
\item Аналогично, слабую топологию можно задать системой окрестностей 
\[
\left\{ \left(\mathrm{d}_{n},T,\varepsilon\right)\right\} _{T\in\mathcal{A},n\in\mathbb{N},\varepsilon>0},
\]
где 
\[
\mathrm{d}_{n}\left(T,S\right)=\sum_{i\leqslant n}\frac{1}{2^{i}}\left(\mu\left(TA_{i}\bigtriangleup SA_{i}\right)+\mu\left(T^{-1}A_{i}\bigtriangleup S^{-1}A_{i}\right)\right).
\]
\item Множество $\mathcal{A}$ --- неполно относительно метрики $\mathrm{a}$,
но является полным сепарабельным пространством относительно метрики
$\mathrm{d}$.
\item Для всех $T,S\in\mathcal{A}$, имеет место неравенство
\[
2\geq\mathrm{d}\left(T,S\right)\geqslant\mathrm{a}\left(T,S\right).
\]
\item $\mathcal{A}$ является топологической группой (то есть операции взятия
обратного и умножения непрерывны).
\end{itemize}
В силу третьего свойства, метрика $\mathrm{d}$ предпочтительней для
пространства $\mathcal{A}$, она и называется \emph{метрикой слабой
топологии}.

\subsection*{Метрики в $\mathcal{A}_{\mathcal{G}}$}

Прежде всего докажем непрерывность $\mathcal{G}$-действий. 
\begin{claim}
Действие группы $\mathcal{G}$ непрерывно относительно слабой топологии.
\end{claim}
\begin{proof}
Зафиксируем $\varepsilon>0$, действие $T$, произвольное $g\in\mathcal{G}$
и окрестность $\left(\mathrm{a}_{n},T^{g},\varepsilon\right)$. Для
каждой пары чисел $i,j\leq n$ из непрерывности отображения $\mu\left(T^{g}A_{i}\cap A_{j}\right)$
следует существование такой открытой окрестности $\mathcal{O}_{i,j}\subset\mathcal{G}$
элемента $g$, что $\mu\left(T^{g}A_{i}\cap A_{j}\right)\overset{\varepsilon}{\sim}\mu\left(T^{h}A_{i}\cap A_{j}\right)$
для всех $h\in\mathcal{O}_{i,j}$. Тогда, при $h\in\cap_{i,j\leq n}\mathcal{O}_{ij}$,
имеем $T^{h}\in\left(\mathrm{a}_{n},T^{g},\varepsilon\right)$.
\end{proof}
Метрику в $\mathcal{A}_{\mathcal{G}}$ определим функцией 
\[
\mathrm{d}_{\mathcal{G}}\left(T,S\right)=\sum_{i}\frac{1}{2^{i}}\sup_{g\in K_{i}}\mathrm{d}\left(T^{g},S^{g}\right).
\]

\begin{thm}
Множество действий группы $\mathcal{G}$ является полным пространством
относительно метрики $\mathrm{d}_{\mathcal{G}}$.
\end{thm}
\begin{proof}
Проверим, что $\mathrm{d}_{\mathcal{G}}$ действительно метрика. Действительно,
функция определена везде, так как супремумы в правой части берутся
от равномерно ограниченных числовых множеств. Если расстояние $\mathrm{d}_{\mathcal{G}}\left(T,S\right)$
равно нулю, то $T^{g}=S^{g}$ для всех $i$ и $g\in K_{i}$, следовательно
для всех $g\in\mathcal{G}$. Далее, выполняется и неравенство треугольника,
так как для любого $K\in\left\{ K_{i}\right\} $, имеем
\[
\sup_{g\in K}\mathrm{d}\left(S^{g},T^{g}\right)\leq\sup_{g\in K}\left(\mathrm{d}\left(T^{g},U^{g}\right)+\mathrm{d}\left(U^{g},S^{g}\right)\right)\leq\sup_{g\in K}\mathrm{d}\left(T^{g},U^{g}\right)+\sup_{g\in K}\mathrm{d}\left(U^{g},S^{g}\right).
\]

Пусть $\left\{ T_{i}\right\} $ --- фундаментальная последовательность
действий группы $\mathcal{G}$. Зафиксируем любое $g\in\cup_{l}K_{l}$
и пусть $k$ --- минимальный номер, для которого $g\in K_{k}$. Тогда
в $\mathcal{A}$ фундаментальна последовательность $\left\{ T_{i}^{g}\right\} $,
так как
\[
\mathrm{d}\left(T_{i}^{g},T_{j}^{g}\right)\leq2^{k}\sum_{l\geq k}\frac{1}{2^{l}}\mathrm{d}\left(T_{i}^{g},T_{j}^{g}\right)\leq2^{k}\sum_{i}\frac{1}{2^{i}}\sup_{g\in K_{i}}\mathrm{d}\left(T_{i}^{g},T_{j}^{g}\right)=2^{k}\mathrm{d}\left(T_{i},T_{j}\right).
\]

Для произвольного $g\in\mathcal{G}$, $g=g_{1}...g_{n}$, где $\left\{ g_{i}\right\} \subset\cup_{l}K_{l}$,
определим преобразование $T^{g}$ равенством 
\[
T^{g}=T^{g_{1}}T^{g_{2}}...T^{g_{n}}.
\]
 Определение $T^{g}$ корректно, так как из непрерывности операции
умножения следует, что $T_{i}^{g}\to T^{g}$. 

Покажем, что $\left\{ T^{g}\right\} _{g\in\mathcal{G}}$ --- действие
группы $\mathcal{G}$. 

Проверим групповое свойство.

Для всех $h,g\in\mathcal{G}$ имеем $T_{i}^{g}\to T^{g}$,$T_{i}^{h}\to T^{h}$,
следовательно, исходя из группового свойства в $\mathcal{A}$, имеем
\[
T_{i}^{g}T_{i}^{h}\to T^{g}T^{h}.
\]
В то же время 
\[
T_{i}^{g}T_{i}^{h}=T_{i}^{gh}\to T^{gh},
\]
то есть $T^{gh}=T^{g}T^{h}$. 

Осталось проверить непрерывность отображения $g\mapsto T^{g}$ в единице
$e$ группы $\mathcal{G}$. Не теряя общности, можно считать, что
$e$ содержится в некоторой открытой окрестности $\mathcal{O}_{1}\subset K_{1}$.

Пусть $A,B\in\Sigma$ и $\varepsilon>0$. Рассмотрим такой номер $l$,
что 
\[
\mu\left(T_{i}^{h}A\cap B\right)\overset{\varepsilon}{\sim}\mu\left(T_{l}^{h}A\cap B\right)
\]
 для всех $h\in K_{1}$ и $i>l$. Такой номер существует, так как
последовательность $\left\{ T_{i}\right\} $ фундаментальна. 

Поскольку $T_{i}^{h}\to T^{h}$ при $i\to\infty$, отсюда следует,
что 
\[
\mu\left(T^{h}A\cap B\right)\overset{2\varepsilon}{\sim}\mu\left(T_{l}^{h}A\cap B\right)
\]
 для всех $h\in K_{1}$.

Тогда, для всех $h\in\mathcal{O}_{1}$, удовлетворяющих открытому
условию 
\[
\mu\left(T_{l}^{h}A\cap B\right)\overset{\varepsilon}{\sim}\mu\left(A\cap B\right),
\]
 имеем 
\[
\left|\mu\left(A\cap B\right)-\mu\left(T^{h}A\cap B\right)\right|\leq\left|\mu\left(A\cap B\right)-\mu\left(T_{l}^{e}A\cap B\right)\right|+
\]
\[
+\left|\mu\left(T_{l}^{e}A\cap B\right)-\mu\left(T_{l}^{h}A\cap B\right)\right|+\left|\mu\left(T_{l}^{h}A\cap B\right)-\mu\left(T^{h}A\cap B\right)\right|<3\varepsilon.
\]

Таким образом, действие $T$ непрерывно. 

Очевидно, что $T$ является пределом последовательности $\left\{ T_{i}\right\} $
так как 
\[
\mathrm{d}\left(T_{i}^{h},T^{h}\right)\to0
\]
равномерно по всем $h\in K_{1}$, $h\in K_{2}$ и так далее.
\end{proof}
Метрика $\mathrm{d_{\mathcal{G}}}$ является равномерным пределом
возрастающей последовательности предметрик 
\[
\mathrm{d}_{\mathcal{G}}^{\left(n\right)}\left(T,S\right)=\sum_{i\leqslant n}\frac{1}{2^{i}}\sup_{g\in K_{i}}\mathrm{d}\left(T^{g},S^{g}\right),
\]

Следовательно базу $\mathrm{d}_{\mathcal{G}}$\-топологии образуют
множества 
\[
\left\{ \left(\mathrm{d}_{\mathcal{G}}^{\left(n\right)},T,\varepsilon\right)\right\} _{T\in\mathcal{A}_{\mathcal{G}},n\in\mathbb{N},\varepsilon>0}.
\]

\begin{thm}
Множество действий группы $\mathcal{G}$ является полным сепарабельным
пространством относительно метрики $\mathrm{d}_{\mathcal{G}}$.
\end{thm}
\begin{proof}
Полнота пространства $\mathcal{A}_{\mathcal{G}}$ установлена в предыдущей
теореме. Покажем его сепарабельность. 

Так как предметрики $\mathrm{d}_{\mathcal{G}}^{\left(n\right)}\left(T,S\right)$
сходятся к $\mathrm{d}_{\mathcal{G}}\left(T,S\right)$ равномерно
при $n\to\infty$, достаточно найти счетную $\left(\mathrm{d}_{\mathcal{G}}^{\left(n\right)},4\varepsilon\right)$\-сеть
при любых фиксированных $n\in\mathbb{N},\varepsilon>0$.

Через $L$ обозначим любое счетное всюду плотное подмножество в $K=\cup_{i\leq n}K_{i}$,
через $\left\{ T_{i}\right\} $ --- счетное всюду плотное множество
в $\mathcal{A}$.

Зафиксируем произвольное действие $S$. Для $f\in L$ и $\varepsilon>0$
рассмотрим множество 
\[
\left\{ g\in\mathcal{G}\,:\,\mathrm{d}\left(S^{g},S^{f}\right)<\varepsilon\right\} .
\]

Эти множества задают открытое покрытие $K$. Так как $K$ --- компакт,
то из него можно выделить конечное подпокрытие с некоторыми центрами
$l=\left(l_{1},...,l_{k}\right)\subset L^{k}$. В этом случае скажем,
что $S\in K\left(k,l,\varepsilon\right)$.

Рассмотрим счетный набор множеств, зависящих от $k,\varepsilon$ и
векторов $l\in L^{k},m\in\mathbb{N}^{k}$: 

\[
\mathcal{O}(k,l,m,\varepsilon)=\left\{ T\in K\left(k,l,\varepsilon\right)\ :\ \mathrm{d}\left(T^{l_{i}},T_{m_{i}}\right)<\varepsilon,\ i=1,...,k\right\} .
\]

Очевидно, что каждое действие лежит как минимум в одном из этих множеств.

Если два действия $T,S$ лежат в одном множестве $\mathcal{O}(k,l,m,\varepsilon)$,
то $\mathrm{d}_{\mathcal{G}}^{\left(n\right)}$-расстояние между ними
не превышает $4\varepsilon$, так как для любого $g\in K$ и некоторого
$l_{i}$, имеем 
\[
\mathrm{d}\left(T^{g},S^{g}\right)\leqslant\mathrm{d}\left(T^{g},T^{l_{i}}\right)+\mathrm{d}\left(T^{l_{i}},T_{m_{i}}\right)+\mathrm{d}\left(T_{m_{i}},S^{l_{i}}\right)+\mathrm{d}\left(S^{l_{i}},S^{g}\right)<4\varepsilon.
\]
Таким образом, взяв по одному действию в каждом непустом множестве
$\mathcal{O}(k,l,m,\varepsilon)$, мы получаем счетную $\left(\mathrm{d}_{\mathcal{G}}^{\left(n\right)},4\varepsilon\right)$-сеть
в $\mathcal{A}_{\mathcal{G}}$.
\end{proof}

\subsection*{Пространство $\mathcal{M}_{\mathcal{G},\Gamma}$}

Метрика поводок\-топологии на множестве $\mathcal{M}_{\mathcal{G},\Gamma}$
задается формулой 
\[
\mathrm{m}_{\mathcal{G},\Gamma}\left(T,S\right)=\mathrm{d}_{\mathcal{G}}\left(T,S\right)+\sup_{g\in\Gamma}\mathrm{a}\left(T^{g},S^{g}\right).
\]

\begin{lem}
$\mathcal{M}_{\mathcal{G},\Gamma}$ --- сепарабельное пространство. 
\end{lem}
\begin{proof}
Достаточно показать, что для любого $k$ и произвольного $\varepsilon>0$
в $\mathcal{M}_{\mathcal{G},\Gamma}$ существует счетная $\left(\mathrm{m}_{\mathcal{G},\Gamma},\varepsilon\right)$\-сеть. 

Пусть $\left\{ G_{i}\right\} $ --- возрастающая последовательность
открытых множеств с компактным замыканием, покрывающая $\mathcal{G}$. 

Положим 
\[
\mathcal{Q}_{i}=\left\{ T\in\mathcal{M}_{\mathcal{G},\Gamma}\mid g\in\left(\Gamma\setminus G_{i}\right)\Rightarrow\mathrm{a}\left(T^{g},\Theta\right)<\frac{\varepsilon}{2}\right\} ,
\]

и покажем, что любое $\Gamma$-перемешивающее преобразование $T$
лежит в одном из этих множеств. По определению $\Gamma$-перемешивания,
существует такой компакт $K$, что при $g\in\left(\Gamma\setminus K\right)$,
имеем $\mathrm{a}\left(T^{g},\Theta\right)<\frac{\varepsilon}{2}$.
Так как $\left\{ G_{i}\right\} $ --- покрытие $K$, то существует
такой номер $i$, для которого $K\subset G_{i}$. Тогда 
\[
g\in\left(\Gamma\setminus G_{i}\right)\Rightarrow g\in\left(\Gamma\setminus K\right)\Rightarrow\mathrm{a}\left(T^{g},\Theta\right)<\frac{\varepsilon}{2},
\]
 то есть $T\in\mathcal{Q}_{i}$. 

Рассмотрим метрику слабой топологии $\mathrm{d}_{\mathcal{G}}^{\prime}$,
порожденную множествами $\left\{ \overline{G_{i}}\cup K_{j}\right\} _{j}$
(здесь $\overline{G_{i}}$ --- замыкание множества $G_{i}$). Поскольку
это --- метрика слабой топологии в $\mathcal{A}_{\mathcal{G}}$,
существует счетное покрытие $\mathcal{A}_{\mathcal{G}}$ множествами
$\left\{ \mathcal{U}_{i,p}\right\} _{p\in\mathbb{N}}$, диаметры которых
в метрике $\mathrm{d}_{\mathcal{G}}^{\prime}$ меньше $\varepsilon$.
Воспользовавшись неравенствами $\mathrm{d}_{\mathcal{G}}\leq\mathrm{d}_{\mathcal{G}}^{\prime}$
и $\mathrm{a}\leq\mathrm{d}$, получаем, что при $T,S\in\mathcal{U}_{i,p}$,
имеем 
\[
\mathrm{d}_{\mathcal{G}}\left(T,S\right)<\varepsilon.
\]
Кроме того, 
\[
\sup_{g\in G_{i}}\mathrm{a}\left(T^{g},S^{g}\right)\leq\sum_{i}\frac{1}{2^{i}}\sup_{g\in G_{i}}\mathrm{d}\left(T^{g},S^{g}\right)\leq\mathrm{d}_{\mathcal{G}}^{\prime}\left(T,S\right)<\varepsilon.
\]

Далее, в каждом непустом множестве вида 
\[
\mathcal{Q}_{i}\cap\mathcal{U}_{i,p}\cap\mathcal{M}_{\mathcal{G},\Gamma}
\]
возьмем один элемент. Тогда, каждое $\Gamma$\-перемешивающее действие
лежит в одном из этих множеств. 

Для двух элементов $T,S\in\mathcal{Q}_{i}\cap\mathcal{U}_{i,p}\cap\mathcal{M}_{\mathcal{G},\Gamma}$,
имеем 
\[
\sup_{g\in\Gamma}\mathrm{a}\left(T^{g},S^{g}\right)+\mathrm{d}_{\mathcal{G}}\left(T,S\right)<\max\left\{ \sup_{g\in G_{i}}\mathrm{a}\left(T^{g},S^{g}\right),\sup_{g\notin G_{i}}\mathrm{a}\left(T^{g},S^{g}\right)\right\} +\varepsilon\leqslant2\varepsilon.
\]
 Таким образом, в пространстве $\mathcal{M}_{\mathcal{G},\Gamma}$
имеется счетная $\left(\mathrm{m},2\varepsilon\right)$-сеть. Следовательно,
оно сепарабельно.
\end{proof}
\begin{thm}
\label{GL}$\mathcal{M}_{\mathcal{G},\Gamma}$ --- полное сепарабельное
пространство.
\end{thm}
\begin{proof}
В силу предыдущей леммы достаточно доказать полноту. Любая $\mathrm{m}$\-фундаментальная
последовательность $\left\{ T_{i}\right\} $ является также $\mathrm{d}$\-фундаментальной
в $\mathcal{A}_{\mathcal{G}}$. Пусть $\mathrm{d}$\-предел этой
последовательности --- некоторое $\mathcal{G}$\-действие $T$. 

Возьмем произвольное $\varepsilon>0$.

Поскольку последовательность $\left\{ T_{i}\right\} $ фундаментальна,
найдется такое $i$, что при $b,c>i$ и любом $g\in\Gamma$, имеем
$\mathrm{a}\left(T_{b}^{g},T_{c}^{g}\right)<\varepsilon.$ Далее,
$T_{b}^{g}$ слабо сходится к $T^{g}$, значит, существует такое $n\left(g\right)$,
что $\mathrm{a}\left(T_{n\left(g\right)}^{g},T^{g}\right)<\varepsilon$. 

Тогда при $b>i$, 
\[
\sup_{g\in\Gamma}\mathrm{a}\left(T^{g},T_{b}^{g}\right)\leqslant\sup_{g\in\Gamma}\left(\mathrm{a}\left(T^{g},T_{n\left(g\right)}^{g}\right)+\mathrm{a}\left(T_{b}^{g},T_{n\left(g\right)}^{g}\right)\right)<2\varepsilon.
\]

Таким образом, $T$ --- $\mathrm{m}$\-предел последовательности
$\left\{ T_{i}\right\} $. 

Далее, покажем, что предельное действие $T$ перемешивает на $\Gamma$.
Возьмем любое $n\in\mathbb{N}$ и такое $j$, что 
\[
\mathrm{m}_{\mathcal{G},\Gamma}\left(T,T_{j}\right)<\varepsilon.
\]
 Так как $T_{j}$ является $\Gamma$\-перемешивающим действием, существует
такое компактное множество $G$, что при $g\in\Gamma\setminus G$
имеем 
\[
\mathrm{a}\left(T_{j}^{g},\Theta\right)<\varepsilon.
\]
 Тогда, для $g\in\Gamma\setminus G$, имеет место оценка 
\[
\mathrm{a}\left(T^{g},\Theta\right)\leqslant\mathrm{a}\left(T^{g},T_{j}^{g}\right)+\mathrm{a}\left(T_{j}^{g},\Theta\right)<2\varepsilon.
\]
В силу произвольности $\varepsilon$ заключаем, что $T$ перемешивает
на $\Gamma$. 
\end{proof}
В заключении дадим достаточное условие для того, чтобы множесво $\mathcal{M}_{\mathcal{G},\Gamma}$
совпадало со множеством всех сильно перемешивающих $\mathcal{G}$\-действий.

Множество $\Gamma$ назовем $H$\-\emph{сетью} в $\mathcal{G}$,
если $\mathcal{G}=H\Gamma$.
\begin{claim}
Пусть $H$ --- компактное подмножество $\mathcal{G}$ и $\Gamma$
--- $H$\-сеть. Тогда любое $\mathcal{G}$\-действие, перемешивающее
на $\Gamma$, перемешивает на всей группе $\mathcal{G}$.
\end{claim}
\begin{proof}
Зафиксируем $\varepsilon>0$, $A,B\in\Sigma$ и $T\in\mathcal{M}_{\mathcal{G},\Gamma}$.
Тогда существует такое компактное $G$, что при $g\in\Gamma\setminus G$
имеем
\[
\mu\left(T^{g}C\cap D\right)\overset{\varepsilon}{\sim}\mu\left(C\right)\mu\left(D\right),
\]
где $C,D\in\left\{ A\right\} \cup\left\{ T^{h^{-1}}B\right\} _{h\in H}$.
Любой элемент $\mathcal{G}$ представляется в виде $hg$, где $h\in H$
и $g\in\Gamma$. При $hg\notin HG$, имеем 
\[
\mu\left(T^{hg}A\cap B\right)=\mu\left(T^{g}A\cap T^{h^{-1}}B\right)\overset{\varepsilon}{\sim}\mu\left(A\right)\mu\left(B\right).
\]
 В силу произвольности $\varepsilon$, $A$ и $B$, действие $T$
является $\Gamma$-перемешивающим.
\end{proof}

\section{Топологии и дополнительные метрики}

Цель этого параграфа состоит в получении более проверяемых условий
близости действий в поводок\-топологии.

Аналогично метрике $\mathrm{d}_{\mathcal{G}}$, в пространстве $\mathcal{A}_{\mathcal{G}}$
топологию можно задать метрикой
\[
\mathrm{a}_{\mathcal{G}}\left(T,S\right)=\sum_{i}\frac{1}{2^{i}}\sup_{g\in K_{i}}\mathrm{a}\left(T^{g},S^{g}\right),
\]
которая является равномерным пределом последовательности предметрик
\[
\mathrm{a}_{\mathcal{G}}^{\left(n\right)}\left(T,S\right)=\sum_{i\leqslant n}\frac{1}{2^{i}}\sup_{g\in K_{i}}\mathrm{a}\left(T^{g},S^{g}\right).
\]
 
\begin{claim}
\label{=000443=000442=000432:=000020=000442=00043E=00043F=00043E=00043B=00043E=000433=000438=000438=000020=00043C=000435=000442=000440=000438=00043A=000020=000430=000020=000438=000020=000434=000020=000441=00043E=000432=00043F=000430=000434=000430=00044E=000442=000020=000434=00043B=00044F=000020=000433=000440=000443=00043F=00043F}Топологии
метрик $\mathrm{a}_{\mathcal{G}}$ и $\mathrm{d}_{\mathcal{G}}$ совпадают.
\end{claim}
\begin{proof}
Так как $\mathrm{a}\left(T^{g},S^{g}\right)\leqslant\mathrm{d}\left(T^{g},S^{g}\right)$
при $g\in\mathcal{G}$, множество $\left(\mathrm{d}_{\mathcal{G}},T,\varepsilon\right)$
содержится во множестве $\left(\mathrm{a}_{\mathcal{G}},T,\varepsilon\right)$.
Следовательно, $\mathrm{d}_{\mathcal{G}}$-топология не слабее $\mathrm{a}_{\mathcal{G}}$\-топологии.

Нам нужно показать, что при фиксированных $\varepsilon$ и $n$ найдется
такое $\delta$, что 
\[
\left(\mathrm{a}_{\mathcal{G}}^{\left(n\right)},T,\delta\right)\subset\left(\mathrm{d}_{\mathcal{G}}^{\left(n\right)},T,\varepsilon\right).
\]

Топологии метрик $\mathrm{a}$ и $\mathrm{d}$ совпадают, поэтому
для каждого $g\in\cup_{i\leqslant n}K_{i}$ существует число $\delta\left(g\right)$
такое, что для любого $U\in\mathcal{A}$, имеем 
\[
\mathrm{a}\left(T^{g},U\right)<\delta\left(g\right)\Rightarrow\mathrm{d}\left(T^{g},U\right)<\frac{\varepsilon}{2}.
\]
Возьмем окрестность элемента $g$, любой элемент $h$ которой удовлетворяет
условиям $\mathrm{d}\left(T^{g},T^{h}\right)<\frac{\varepsilon}{2}$
и $\mathrm{a}\left(T^{g},T^{h}\right)<\frac{\delta\left(g\right)}{2}$. 

Так как $\cup_{i\leqslant n}K_{i}$ --- компакт, существует его конечное
покрытие выбранными окрестностями с центрами в некоторых точках $\left\{ g_{j}\right\} $.
Положим $\delta=\min_{j}\delta\left(g_{j}\right)$.

Если $S\in\left(\mathrm{a}_{\mathcal{G}}^{\left(n\right)},T,\delta\right)$,
то для любого $g\in\cup_{i\leqslant n}K_{i}$ и некоторого $j$, имеем
\[
\mathrm{a}\left(T^{g_{j}},S^{g}\right)\leqslant\mathrm{a}\left(T^{g_{j}},T^{g}\right)+\mathrm{a}\left(T^{g},S^{g}\right)<\delta\left(g_{j}\right).
\]
Следовательно, 
\[
\mathrm{d}\left(T^{g},S^{g}\right)\leqslant\mathrm{d}\left(T^{g},T^{g_{j}}\right)+\mathrm{d}\left(T^{g_{j}},S^{g}\right)<\varepsilon,
\]
 и $S\in\left(\mathrm{d}_{\mathcal{G}}^{\left(n\right)},T,\varepsilon\right)$. 
\end{proof}
\begin{cor}
Поводок\-топологию можно задать с помощью метрики 
\[
\mathrm{w}_{\mathcal{G},\Gamma}\left(T,S\right)=\mathrm{a}_{\mathcal{G}}\left(T,S\right)+\sup_{g\in\Gamma}\mathrm{a}\left(T^{g},S^{g}\right).
\]
\end{cor}
\begin{proof}
Пусть $\varepsilon>0$. Выберем $\delta>0$ такое, что 
\[
\left(\mathrm{a}_{\mathcal{G}},T,\delta\right)\subset\left(\mathrm{d}_{\mathcal{G}},T,\varepsilon\right).
\]
Тогда 
\[
\left(\mathrm{w}_{\mathcal{G},\Gamma},T,\delta\right)\subset\left(\mathrm{m}_{\mathcal{G},\Gamma},T,\varepsilon\right).
\]
Так как $\mathrm{w}_{\mathcal{G},\Gamma}\leq\mathrm{m}_{\mathcal{G},\Gamma}$
верно и включение 
\[
\left(\mathrm{m}_{\mathcal{G},\Gamma},T,\varepsilon\right)\subset\left(\mathrm{w}_{\mathcal{G},\Gamma},T,\varepsilon\right),
\]
то есть топологии совпадают. 
\end{proof}
Метрика $\mathrm{w}$ является пределом возрастающей последовательности
предметрик 
\[
\mathrm{w}_{\mathcal{G},\Gamma}^{\left(n\right)}\left(T,S\right)=\mathrm{a}_{\mathcal{G}}^{\left(n\right)}\left(T,S\right)+\sup_{g\in\Gamma}\mathrm{a}\left(T^{g},S^{g}\right),
\]
каждая из которых, в свою очередь, является пределом возрастающей
последовательности предметрик 
\[
\mathrm{w}_{\mathcal{G},\Gamma}^{\left(n,k\right)}\left(T,S\right)=\mathrm{a}_{\mathcal{G}}^{\left(n,k\right)}\left(T,S\right)+\sup_{g\in\Gamma}\mathrm{a}_{k}\left(T^{g},S^{g}\right),
\]
где 
\[
\mathrm{a}_{\mathcal{G}}^{\left(n,k\right)}\left(T,S\right)=\sum_{i\leqslant n}\frac{1}{2^{i}}\sup_{g\in K_{i}}\mathrm{a}_{k}\left(T^{g},S^{g}\right).
\]
Таким образом, базу окрестностей поводок\-топологии образуют и множества
\[
\left\{ \left(\mathrm{w}_{\mathcal{G},\Gamma}^{\left(n,k\right)},T,\varepsilon\right)\right\} _{n,k\in\mathbb{N},\varepsilon>0,T\in\mathcal{M}_{\mathcal{G},\Gamma}}.
\]

Наконец, зададим еще одно семейство предметрик в $\mathcal{M}_{\mathcal{G},\Gamma}$
\[
\mathrm{s}^{\left(n,k\right)}\left(T,S\right)=\sup_{g\in\Gamma\cup\left(\cup_{i\leq n}K_{i}\right)}\mathrm{a}_{k}\left(T^{g},S^{g}\right).
\]

Параметры $\mathcal{G},\Gamma$ для этих предметик писать не будем,
так как они заданы только в одном из рассматриваемых пространств.
\begin{thm}
Базу поводок\-топологии образуют множества 
\[
\left\{ \left(\mathrm{s}^{\left(n,k\right)},T,\varepsilon\right)\right\} _{n,k\in\mathbb{N},\varepsilon>0,T\in\mathcal{M}_{\mathcal{G},\Gamma}}.
\]
\end{thm}
\begin{proof}
Достаточно заметить, что эквивалентные предметрики имеют одинаковую
базу топологии и 
\[
\mathrm{w}_{\mathcal{G},\Gamma}^{\left(n,k\right)}\leq\mathrm{s}^{\left(n,k\right)}\leq2^{n}\mathrm{w}_{\mathcal{G},\Gamma}^{\left(n,k\right)}.
\]
\end{proof}

\section{Слабое рохлинское свойство}

Множество действий вида $U^{-1}TU$, где $U\in\mathcal{A}$, называется
\emph{орбитой} точки $T$. 

Множество действий называется \emph{массивным}, если оно является
всюду плотным пересечением счетного числа открытых множеств. Массивность
множества действий с плотной орбитой называется \emph{слабым рохлинским
свойством}. Мы покажем, что пространство $\mathcal{M}_{\mathcal{G},\Gamma}$
этим свойством обладает. 
\begin{claim}
\label{=00041B=000435=00043C=00043C=000430:=000441=00043E=00043F=000440=00044F=000436=000435=00043D=000438=000435=000020=000434=000020=00043C=000020=00043D=000435=00043F=000440=000435=000440=00044B=000432=00043D=000430=00044F=000020=000433=000440=000443=00043F=00043F=00044B}Для
любого $\Gamma$\-перемешивающего $\mathcal{G}$\-действия $T$
отображение $U\mapsto U^{-1}TU$ из $\mathcal{A}_{\mathcal{}}$ в
$\mathcal{M}_{\mathcal{G},\Gamma}$ непрерывно.
\end{claim}
\begin{proof}
Так как $\mathcal{A}$ --- топологическая группа, достаточно показать
непрерывность отображения в единице $I$ этой группы. Для каждого
$\varepsilon>0$, любых $n,k$, зафиксируем $\left(\mathrm{s}^{\left(n,k\right)},T,\varepsilon\right)$-окрестность
действия $T$.

Пусть $\mathcal{V}\left(I\right)\subset\mathcal{A}$ такое открытое
множество, что 
\[
U\in\mathcal{V}\left(I\right)\Rightarrow\forall i\leq k,\mu\left(A_{i}\Delta UA_{i}\right)<\frac{\varepsilon}{2}.
\]
Тогда, для $U\in\mathcal{V}\left(I\right)$, любого $g\in\mathcal{G}$
и $i,j\leq k$, имеем 
\[
\left|\mu\left(T^{g}A_{i}\cap A_{j}\right)-\mu\left(U^{-1}T^{g}UA_{i}\cap A_{j}\right)\right|=
\]

\[
=\left|\mu\left(T^{g}A_{i}\cap A_{j}\right)-\mu\left(T^{g}UA_{i}\cap UA_{j}\right)\right|\overset{\varepsilon}{\sim}
\]
\[
\overset{\varepsilon}{\sim}\left|\mu\left(T^{g}A_{i}\cap A_{j}\right)-\mu\left(T^{g}A_{i}\cap A_{j}\right)\right|=0.
\]

Следовательно, для всех $g\in\mathcal{G}$, имеем 
\[
\mathrm{a}_{k}\left(U^{-1}T^{g}U,T^{g}\right)<\varepsilon.
\]
Тогда
\[
\mathrm{s}^{\left(n,k\right)}\left(U^{-1}TU,I^{-1}TI\right)=\sup_{g\in\Gamma\cup\left(\cup_{i\leq n}K_{i}\right)}\mathrm{a}_{k}\left(U^{-1}T^{g}U,T^{g}\right)<\varepsilon,
\]

Значит любая точка из окрестности $\mathcal{V}$ переходит в точку
из $\left(\mathrm{s}^{\left(n,k\right)},T,\varepsilon\right)$\-окрестности
$T$.
\end{proof}
\begin{lem}
\label{=000443=000442=000432:=00043D=000435=00043F=000440=000435=000440=00044B=000432=00043D=00043E=000441=000442=00044C=000020=000441=00043E=00043F=000440=00044F=000436=000435=00043D=000438=00044F}При
фиксированном $U\in\mathcal{A}$, отображение $T\mapsto U^{-1}TU$
из пространства $\mathcal{M}_{\mathcal{G},\Gamma}$ в себя непрерывно.
\end{lem}
\begin{proof}
Для произвольных $\varepsilon>0$ и $n$ найдем такое число $k$,
что 
\[
S\in\left(\mathrm{a}_{k},T^{g},\frac{\varepsilon}{2}\right)\Rightarrow U^{-1}SU\in\left(\mathrm{a}_{n},U^{-1}T^{g}U,\varepsilon\right),
\]
для всех $g\in\mathcal{G}$. 

Число $k$ выберем таким, что для каждого $i\leqslant n$ найдется
$j\leqslant k$ для которого 
\[
\mu\left(UA_{i}\bigtriangleup A_{j}\right)<\frac{\varepsilon}{8}.
\]

Тогда для произвольного $g\in\mathcal{G}$ и $S\in\left(\mathrm{a}_{k},T^{g},\frac{\varepsilon}{2}\right)$,
имеем, 
\[
\mathrm{a}_{n}\left(U^{-1}S^{g}U,U^{-1}T^{g}U\right)\leqslant
\]
\[
\leqslant\sup_{i,j\leqslant n}\left|\mu\left(U^{-1}S^{g}UA_{i}\cap A_{j}\right)-\mu\left(U^{-1}T^{g}UA_{i}\cap A_{j}\right)\right|=
\]
\[
=\sup_{i,j\leqslant n}\left|\mu\left(S^{g}UA_{i}\cap UA_{j}\right)-\mu\left(T^{g}UA_{i}\cap UA_{j}\right)\right|.
\]
Пусть $A_{a}$ и $A_{b}$ --- множества из набора $q_{k}$, приближающие
с точностью до $\frac{\varepsilon}{8}$ множества $UA_{i}$ и $UA_{j}$
соответственно. Тогда 
\[
\left|\mu\left(SUA_{i}\cap UA_{j}\right)-\mu\left(T^{g}UA_{i}\cap UA_{j}\right)\right|<
\]
\[
<\left|\mu\left(SA_{a}\cap A_{b}\right)-\mu\left(T^{g}A_{a}\cap A_{b}\right)\right|+\frac{\varepsilon}{2}<\varepsilon.
\]
Следовательно, 
\[
\mathrm{a}_{n}\left(U^{-1}SU,U^{-1}T^{g}U\right)<\varepsilon.
\]

Тогда, для выбранных $\varepsilon,n,k$ , и любого натурального $i$,
имеем 
\[
S\in\left(\mathrm{s}^{\left(i,k\right)},T,\frac{\varepsilon}{2}\right)\Rightarrow U^{-1}SU\in\left(\mathrm{s}^{\left(i,n\right)},U^{-1}TU,\varepsilon\right).
\]
\end{proof}
Нам понадобятся известный факт, что пространство $\left(X,\Sigma,\mu\right)$
изоморфно любой своей не более чем счетной тензорной степени, то есть
пространству $\left(\times_{i}X,\otimes_{i}\Sigma,\otimes_{i}\mu\right)$.
Если набор действий $\left\{ T_{i}\right\} $ содержится в $\mathcal{M}_{\mathcal{G},\Gamma}$,
то $\times_{i}T_{i}\in\mathcal{M}_{\mathcal{G},\Gamma}$.
\begin{lem}
\label{=00043B=000435=00043C:=000020=000434=000435=00043A=000430=000440=000442=00043E=000432=00043E=000020=00043F=000440=00043E=000438=000437=000432=000435=000434=000435=00043D=000438=000435=000020=00043F=000440=000438=000431=00043B=000438=000436=000430=000435=000442=000020=00043B=00044E=000431=00043E=000439=000020=00043C=00043D=00043E=000436=000438=000442=000435=00043B=00044C}Пусть
$T,S$ --- два $\Gamma$-перемешивающих $\mathcal{G}$\-действия.
Тогда, сопряженные прямому произведению $T\times S$, с любой точностью
приближают $T$ в поводок\-топологии.
\end{lem}
\begin{proof}
Достаточно показать, что для любых $k,\varepsilon>0$ найдется сохраняющий
меру изоморфизм $v:X\rightarrow X\times X$ такой, что 
\[
v^{-1}\left(T^{g}\times S^{g}\right)v\in\left(\mathrm{a}_{k},T^{g},\varepsilon\right),
\]
 для любого $g\in\mathcal{G}$.

Возьмем $v$ таким, что $vA_{i}=A_{i}\times X$ для всех $i\leqslant k$.

Тогда, для любых $g$ и $i,j\leqslant k$, имеем
\[
\mu\left(v^{-1}\left(T^{g}\times S^{g}\right)vA_{i}\cap A_{j}\right)=\mu\otimes\mu\left(\left(T^{g}\times S^{g}\right)A_{i}\times X\cap vA_{j}\right)=
\]
\[
=\mu\otimes\mu\left(\left(T^{g}A_{i}\times X\right)\cap\left(A_{j}\times X\right)\right)=\mu\left(T^{g}A_{i}\cap A_{j}\right).
\]
Следовательно, 
\[
\mathrm{a}_{k}\left(v^{-1}\left(T^{g}\times S^{g}\right)v,T\right)=
\]
\[
=\sum_{i,j\leqslant k}\frac{1}{2^{i+j}}\left|\mu\left(v^{-1}\left(T^{g}\times S^{g}\right)vA_{i}\cap A_{j}\right)-\mu\left(T^{g}A_{i}\cap A_{j}\right)\right|=0.
\]
\end{proof}
\begin{claim}
\label{=000443=000442=000432:=000020=000413=00043B=000430=000437=00043D=000435=000440-=000412=000435=000439=000441-=000422=000443=000432=000435=00043D=00043E}Пространство
$\mathcal{M}_{\mathcal{G},\Gamma}$ обладает слабым рохлинским свойством. 
\end{claim}
\begin{proof}
Пусть $\left\{ \mathcal{O}_{i}\right\} $ --- счетная база пространства.
Тогда, множество элементов, сопряженные с которыми всюду плотны, представляется
в виде 
\[
\bigcap_{i}\left\{ T\mid\exists U\in\mathcal{A}:U^{-1}TU\in\mathcal{O}_{i}\right\} .
\]
По лемме \ref{=000443=000442=000432:=00043D=000435=00043F=000440=000435=000440=00044B=000432=00043D=00043E=000441=000442=00044C=000020=000441=00043E=00043F=000440=00044F=000436=000435=00043D=000438=00044F},
это множество является счетным пересечением открытых множеств. Докажем
его плотность. Для этого достаточно найти хотя бы один его элемент
с плотной орбитой. Так как пространство сепарабельно, в нем найдется
счетный всюду плотный набор элементов $\left\{ T_{i}\right\} $. По
предыдущей лемме элементы, сопряженные прямому произведению $\times_{i}T_{i}$,
с любой точностью приближают каждый свой множитель, а значит и любой
элемент пространства.
\end{proof}

\end{document}